\documentclass[a4paper,11pt]{amsart}
\usepackage{amsthm,amsmath,amssymb}
\usepackage[latin1]{inputenc}
\usepackage[english]{babel}
\usepackage{graphicx}
\usepackage[english]{babel}

\DeclareMathOperator{\Aut}{Aut}
\DeclareMathOperator{\Chow}{Chow}

\DeclareMathOperator{\rank}{rk}
\DeclareMathOperator{\ch}{ch}
\DeclareMathOperator{\td}{td}
\DeclareMathOperator{\tr}{tr}
\DeclareMathOperator{\rk}{rk}
\DeclareMathOperator{\sym}{Sym}
\DeclareMathOperator{\Ext}{Ext}
\DeclareMathOperator{\diag}{diag}

\newcommand{\CW}{{\rm CW}}
\newcommand{\CM}{{\rm CM}}

\theoremstyle{plain}
\newtheorem{thm}{Theorem}[section]
\newtheorem{prop}[thm]{Proposition}

\newtheorem{cor}[thm]{Corollary}

\theoremstyle{definition}
\newtheorem{defn}[thm]{Definition}
\theoremstyle{remark}
\newtheorem{ex}[thm]{Example}
\newtheorem{rem}[thm]{Remark}

\begin{document}

\title[Scalar curvature and stability of projective bundles and blowups]{Scalar curvature and asymptotic Chow stability of projective bundles and blowups}
\date{\today}

\author{Alberto Della Vedova}
\address{Fine Hall \\ Princeton University \\ Princeton, NJ 08544 USA}
\email{della@math.princeton.edu}

\author{Fabio Zuddas}
\address{Dipartimento di Matematica \\
         Universit\`a degli Studi di Parma \\ 
         Viale G. P. Usberti, 53/A  \\ 
         43100 Parma (Italy)}
\email{fabio.zuddas@unipr.it}

\begin{abstract}
The holomorphic invariants introduced by Futaki as obstruction to the asymptotic Chow semistability are studied by an algebraic-geometric point of view and are shown to be the Mumford weights of suitable line bundles on the Hilbert scheme of $\mathbb P^n$.

\noindent These invariants are calculated in two special cases. The first is a projective bundle $\mathbb P(E)$ over a curve of genus $g \geq 2$, and it is shown that it is asymptotically Chow polystable (with every polarization) if and only the bundle $E$ is slope polystable. This proves a conjecture of Morrison with the extra assumption that the involved polarization is sufficiently divisible. Moreover it implies that $\mathbb P(E)$ is asymptotically Chow polystable (with every polarization) if and only if it admits a constant scalar curvature K\"ahler metric. The second case is a manifold blown-up at points, and new examples of asymptotically Chow unstable constant scalar curvature K\"ahler classes are given.
\end{abstract}

\maketitle

\section[Introduction]{Introduction}

The existence on a compact complex manifold of a constant scalar curvature K\"ahler (cscK) metric is a central problem in K\"ahler geometry and has been approached with a variety of geometric and analytical methods. While one can look for cscK metrics in arbitrary classes, major interest has been attracted by classes that are first Chern classes of ample line bundles. This is because of a conjecture of Yau \cite{Yau93}, Tian \cite{Tia97,Tia02} and Donaldson \cite{Don02} which relates the existence of a cscK metric in the first Chern class $c_1(A)$ of an ample line bundle $A$ on a manifold $M$, to the GIT stability of the polarized manifold $(M,A)$. More precisely the Yau-Tian-Donaldson conjecture states that a manifold $M$ admits a constant scalar curvature K\"ahler metric in the class $c_1(A)$ if and only if $(M,A)$ is K-polystable. The K-stability has been introduced by Tian \cite{Tia97} and then reformulated in a purely algebraic-geometric way by Donaldson \cite{Don02}. While the ``only if'' part of the conjecture has been proved thanks the works of Tian \cite{Tia97}, Donaldson \cite{Don02}, Arezzo and Pacard \cite{ArePac06}, Stoppa \cite{Sto09} and Mabuchi \cite{Mab08,Mab09}, the ``if'' part is widely open in general and has been proved only in the case of toric surfaces by Donaldson \cite{Don09}, and in the case of projective bundles over a curve thanks to the works of Narasimhan and Seshadri \cite{NarSes65}, Ross and Thomas \cite{RosTho06} and Mabuchi \cite{Mab08, Mab09}.

The K-stability is not the only GIT stability notion related to the existence of cscK metrics. Assuming that $M$ has no nontrivial holomorphic vector fields, Donaldson has proved that the existence of a cscK metric representing $c_1(A)$ implies the asymptotic Chow stability of $(M,A)$ \cite{Don01}. Mabuchi has extended this result removing the hypothesis on the automorphisms and introducing a compatibility hypothesis between the polarization $A$ and $\Aut(M)$ \cite{Mab05}. Finally Futaki, Ono and Sano showed that the Mabuchi's hypothesis is equivalent to the vanishing of certain holomorphic functionals, depending only on the polarization class $c_1(A)$, defined on the space of holomorphic vector fields with zeros of $M$ \cite{FutOnoSan08}. These invariants have been previously introduced by Futaki as an obstruction to asymptotic Chow semistability \cite{Fut04}. Since they generalize the {\it classical} Futaki invariant, we will call them higher Futaki invariants.

As shown by Donaldson \cite{Don02} and Paul and Tian \cite{PauTia06}, the original Futaki invariant has a purely algebraic-geometric nature and it is the Mumford weight of the so-called CM line $\lambda_\CM$ on the Hilbert scheme $\mathcal H$ of $\mathbb P^N$. We will show (section \ref{sec::HiCM}) that the same is true for the higher Futaki invariants.

\begin{thm}
Let $\mathcal H$ be the Hilbert scheme of subschemes of $\mathbb P^N$ with Hilbert polynomial $\chi$. For every $\ell=1,\dots,\deg \chi$ there are $\mathbb Q$-line bundles $\lambda_{\CM, \ell}$ on $\mathcal H$ such that the higher Futaki invariant $F_\ell$ is the Mumford weight of $\lambda_{\CM,\ell}$.
\end{thm}

Furthermore we calculate the Chow weight and the higher Futaki invariants for two classes of manifolds: projective bundles over curves of genus at least two and blowups at finite sets of points.
In the first case (section \ref{sec::prob}) we will show that the higher Futaki invariants are all proportional to the Futaki invariant. This, together with results of Narasimhan and Seshadri \cite{NarSes65} and Mabuchi \cite{Mab05}, is the main ingredient to prove the following

\begin{thm}
A projective bundle $\mathbb P(E)$ over a curve of genus $g \geq 2$ is asymptotically Chow polystable w.r.t. some (and hence any) polarization if and only if $E$ is slope polystable.
\end{thm}

This proves a conjecture of Morrison with the extra assumption that the involved polarization is sufficiently divisible \cite[Conjecture 5.10]{Mor80}. Actually Morrison's conjecture deals also with Hilbert-stability, but we know that Hilbert-stability and Chow stability asymptotically coincide thanks to a result of Mabuchi \cite{Mab08bis}. We notice that one direction of a generalization of the Morrison conjecture to higher dimensional base manifolds has been recently proved by Seyyedali \cite{Sey10} using technics different than ours. On the other hand a recent result due to Apostolov, Calderbank, Gauduchon and T{\o}nnesen-Friedman \cite{ApoCalGauTon09} states that $\mathbb P(E)$ admits a cscK metric in some (and hence any) K\"ahler class if and only if $E$ is slope polystable. Thus we get the following 

\begin{thm}
A projective bundle $\mathbb P(E)$ over a curve of genus $g \geq 2$ is cscK if and only if it is asymptotically Chow polystable w.r.t. some (and hence any) polarization.
\end{thm}

In the case of blowups (section \ref{sec::blup}), we will give a formula for the Chow weight and higher Futaki invariants in terms of data on the base manifold. That formula, together with a result of Arezzo and Pacard \cite{ArePac09}, provides many new examples of asymptotically Chow unstable cscK classes. In particular we give the details in the case of $\mathbb P^2$ blown-up at four points in special position.

\begin{prop}
Let $M$ be the projective plane $\mathbb P^2$ blown up at four points. If all but one are aligned, then $M$ admits an asymptotically Chow unstable cscK polarization. 
\end{prop}

Thus asymptotically Chow unstable cscK classes exist already in dimension two. To the best of our knowledge, the unique example of this kind previously known was the asymptotically Chow unstable K\"ahler-Einstein 7-dimensional toric manifold studied by Ono, Sano and Yotsutani \cite{OnoSanYot09}. Finally it is worth to notice that, thanks to a theorem of Zhang \cite{Zha96}, such manifolds admit no balanced embeddings (see \cite{Don01} for a definition) for each sufficiently high power of the polarization. This is clear from the construction, but it is not a trivial consequence of definition of asymptotic Chow unstability. The point is that one might still have an unbounded sequence of powers of the polarization which are balanceable. However, to the best of the authors' knowledge, examples of such a phenomenon are not know.

\section{acknowledgments}

The first named author would like thank G. La Nave and G. Tian for stimulating discussions. Both the authors would like to thank C. Arezzo for the encouragement  and the interest in this work. The research of the first named author is supported by a Marie Curie International Outgoing Fellowship within the $7^{\rm th}$ European Community Framework Programme (program CAMEGEST, proposal no. 255579).


\section[Higher CM lines]{Higher CM lines}\label{sec::HiCM}
Let $(V,A)$ be an $n$-dimensional polarized variety or scheme. Given a one parameter subgroup $\rho: \mathbb C^* \to {\rm Aut}(V)$ with a linearization on $A$, and denoted by $w(V,A)$ the weight of the $\mathbb C^*$-action induced on $\bigwedge^{\rm top} H^0(V,A)$ we have the following definition, essentially due to Mumford \cite{Mum77}:

\begin{defn}\label{defn::CM-weight}
When $A$ is very ample and without higher cohomology, the (normalized) \emph {Chow weight} of the given action is the rational number $$ \Chow(V,A) = \frac{w(V,A)}{\chi(V,A)} - \frac{b_0(V,A)}{a_0(V,A)}, $$ where $a_0$ and $b_0$ are defined by $w(V,A^k) = b_0(V,A)k^{n+1} + O(k^n)$ and $\chi(V,A^k) = a_0(V,A)k^n + O(k^{n-1})$ as $k\gg 0$.
\end{defn}

Clearly $w(V,A)$ and $b_0(V,A)$ depend on the fixed action $\rho$ and its linearization on $A$. On the other hand $\Chow(V,A)$ depends only on the action and it is insensible of the chosen linearization on $A$ (see Proposition \ref{prop::higher_Futaki} below). In the following we leave these dependences always understood since they will be clear from the context.    

We are interested in the asymptotic behavior of $\Chow(V,A^k)$ when $k$ grows. By the general theory, at least for $k\gg 0$, we have polynomial expansions
\begin{eqnarray}
\label{eq::asym_chi(V,A^k)} \chi(V,A^k) &=& \sum_{\ell=0}^n a_\ell(V,A) k^{n-\ell}, \\
\label{eq::asym_w(V,A^k)} w(V,A^k) &=& \sum_{\ell=0}^{n+1} b_\ell(V,A) k^{n+1-\ell},
\end{eqnarray} 
whence we get easily $$ \Chow(V,A^k) = \frac{b_{n+1}(V,A)}{\chi(V,A^k)} + \frac{a_0(V,A)}{\chi(V,A^k)} \sum_{\ell=1}^n \frac{a_0(V,A)b_\ell(V,A)-b_0(V,A)a_\ell(V,A)}{a_0(V,A)^2}k^{n+1-\ell}. $$

When $V$ is smooth, using equivariant cohomology theory we can express $b_{n+1}(V,A)$ by an integral of an equivariant characteristic class (see the proof of Proposition \ref{prop::higher_Futaki}), and localization formula gives $b_{n+1}(V,A) =0$ \cite[Theorem 5.3.11]{Fut88}. In general the non-vanishing of $b_{n+1}$ turns to be related to non-reduced components of $V$. For example if $V$ is a projective space with a sufficiently non-reduced scheme structure, then a direct calculation of $w(V,A^k)$ shows that $b_{n+1}\neq 0$.

\begin{prop}\label{prop::higher_Futaki}
In the situation above, for $\ell= 1, \dots, n$ let $$ F_\ell(V,A) = \frac{a_0(V,A)b_\ell(V,A)-b_0(V,A)a_\ell(V,A)}{a_0(V,A)^2}. $$ $F_\ell(V,A)$ does not depend on the linearization on $A$, moreover if $V$ is smooth, then $F_\ell(V,A)$ is the $\ell$-th holomorphic invariant discussed in \cite{FutOnoSan08}. 
\end{prop}

\begin{proof}
Changing the linearization of the given $\mathbb C^*$-action on $A$ has the effect to shift the weight $w(V,A^k)$ to $w(V,A^k) + c k \chi(V,A^k)$ for some integral constant $c$. Thus $b_\ell(V,A)$ shifts to $b_\ell(V,A) + c\, a_\ell(V,A)$, and it is easy to see that $F_\ell(V,A)$ is unchanged.

To prove the second assertion we assume $V$ smooth and we use the (equivariant) Hirzebruch-Riemann-Roch theorem in the same way as Donaldson proves that $F_1$ is nothing but the Futaki invariant up to a constant \cite{Don02}. Let $\omega$ be a $S^1$-invariant K\"ahler form representing $c_1(A)$ and fix a moment map $\phi$, so that $\omega + \phi$ represents the equivariant first Chern class of $A$ in the Cartan model of the equivariant cohomology \cite{AtiBot84,BerGetVer04}. Moreover let $\Theta$ be the curvature of the Chern connection $\nabla$ of the K\"ahler metric associated to $\omega$, and let $\Lambda = \nabla_X - L_X$ where $X$ is the holomorphic generator of the given $\mathbb C^*$-action and $L_X$ denotes the Lie derivative. The operator $\Lambda$ defines an endomorphism of the holomorphic tangent bundle and $\Theta + \Lambda$ represents the equivariant curvature of the fixed K\"ahler metric.

By the equivariant Hirzebruch-Riemann-Roch theorem \cite[Theorem 8.2]{BerGetVer04} we know that the character of the induced representation on $\bigoplus_{i\geq 0} H^i(V,A^k)^{(-1)^i}$ (here by the power $-1$ we mean the dual) is given by the integral of a suitable equivariant characteristic forms. On the other hand by ampleness of $A$ we have $H^i(V,A^k)=0$ for all $i>0$ as $k\gg 0$. Thus we get
\begin{eqnarray*} 
w(V,A^k) &=& \left. \frac{d}{dt} \right|_{t=0} \int_V e^{k(\omega + t\phi)} \td(\Theta + t\Lambda) \\
&=& \int_V k\phi e^{k\omega} \td(\Theta) + \int_V e^{k\omega} \left. \frac{d}{dt} \right|_{t=0} \td(\Theta + t\Lambda) \\
&=& \sum_{\ell=0}^n k^{n+1-\ell} \int_V \phi \frac{\omega^{n-\ell}}{(n-\ell)!} \td_\ell(\Theta) + \sum_{\ell=0}^n k^{n-\ell} \int_V \frac{\omega^{n-\ell}}{(n-\ell)!} \td_{\ell+1}(\Theta+\Lambda) \\
&=& \sum_{\ell=0}^n k^{n+1-\ell} \left( \int_V \phi \frac{\omega^{n-\ell}}{(n-\ell)!} \td_\ell(\Theta) + \int_V \frac{\omega^{n+1-\ell}}{(n+1-\ell)!} \td_\ell(\Theta+\Lambda)\right),
\end{eqnarray*}
and $$ \chi(V,A^k) = \int_V e^{k\omega} \td(\Theta) = \sum_{\ell=0}^n k^{n-\ell} \int_V \frac{\omega^{n-\ell}}{(n-\ell)!} td_\ell(\Theta).$$
Let $\hat \phi = b_0(V,A) / a_0(V,A) = \left( \int_V \frac{\omega^n}{n!} \right)^{-1} \int_V \phi \frac{\omega^n}{n!} $ be the average of $\phi$, we have 
$$ F_\ell(V,A) = \frac{\int_V (\phi - \hat \phi) \frac{\omega^{n-\ell}}{(n-\ell)!} \td_\ell(\Theta)}{\int_V \frac{\omega^n}{n!}} + \frac{\int_V \frac{\omega^{n+1-\ell}}{(n+1-\ell)!} \td_\ell(\Theta + \Lambda)}{\int_V \frac{\omega^n}{n!}}, $$ which are essentially the coefficients of polynomial expansion (4) of \cite{FutOnoSan08} if we assume $\hat \phi =0$. 
\end{proof}

We can use the Chow-Mumford weight introduced above and test configurations \cite{Don02,RosTho07} to define the (asymptotic) Chow stability.

\begin{defn}\label{defn::TC}
A \emph{test configuration} $(X, L)\to \mathbb C$ of exponent $k>0$ of a polarized manifold $(M,A)$ consists of a scheme $X$ endowed with a $\mathbb C^*$-action that linearizes on a line bundle $L$ over $X$, and a flat $\mathbb C^*$-equivariant morphism $f: X \to \mathbb C$ (where $\mathbb C^*$ acts on $\mathbb C$ by multiplication) such that $L|_{f^{-1}(0)}$ is ample on $f^{-1}(0)$ and we have $(f^{-1}(1) , L|_{f^{-1}(1)}) \simeq (M,A^k)$.

When $(M,A)$ has a $\mathbb C^*$-action $\rho: \mathbb C^* \to {\rm Aut}(M)$, a test configuration where $X = M \times \mathbb C$ and $\mathbb C^*$ acts on $X$ diagonally through $\rho$ is called \emph{product configuration}. A product configuration endowed with the trivial $\mathbb C^*$-action on $M$ is called \emph{trivial configuration}.
\end{defn} 

Clearly the central fiber $(f^{-1}(0), L|_{f^{-1}(0)})$ of a test configuration is a polarized scheme endowed with a $\mathbb C^*$-action. 

\begin{defn}
The polarized manifold $(M,A)$ is called
\begin{itemize}
\item \emph{asymptotically Chow semistable} if there is a $k_0>0$ such that for each test configuration of exponent $k>k_0$ the Chow weight of the induced action on the central fiber $(f^{-1}(0), L|_{f^{-1}(0)})$ is less than or equal to zero; 
\item \emph{asymptotically Chow polystable} if it is semistable and the Chow weight is zero if and only if we have a product configuration;
\item \emph{asymptotically Chow stable} if it is polystable and every product configuration is trivial.
\item \emph{asymptotically Chow unstable} if it is not semistable.
\end{itemize}
\end{defn}

With the notations introduced above, the main theorem of \cite{Mab05} becomes

\begin{thm}[Mabuchi]\label{thm::Mabuchi}
Let $(M,A)$ be a polarized n-fold admitting a cscK metric representing $c_1(A)$, then $(M,A)$ is asymptotically Chow polystable if and only if for every $\mathbb C^*$-action linearizing on $A$ we have $ F_\ell(M,A)=0$ for all $\ell=1,\dots,n$.
\end{thm}

Paul and Tian \cite{PauTia06} have shown that the Futaki invariant is the Mumford weight (in the sense of GIT) of the line bundle $\lambda_{\CM}$ over the Hilbert scheme $\mathcal H$ of the subschemes of $\mathbb P^N$ with Hilbert polynomial $$\chi(k) = a_0 k^n + a_1 k^{n-1} + \dots + a_n. $$ In the same way the invariant $F_\ell$ is the Mumford weight of the line bundle $\lambda_{\CM,\ell}$ which is constructed as follows.

Let $f: \mathcal U \to \mathcal H$ be the universal (flat) family over $\mathcal H$. By general theory there is an embedding $\iota: \mathcal U \to \mathcal H \times \mathbb P^N$ and $f$ factorizes through it: $f=pr_{\mathcal H} \circ \iota$. Moreover let $\mathcal L=\iota^* \circ pr_{\mathbb P^N}^* \mathcal O_{\mathbb P^N}(1)$ be the relative very ample line bundle on $\mathcal U$ obtained by restriction of the hyperplane bundle on each fiber of $f$.

Thanks to the relative very ampleness of $\mathcal L$, the cohomology $H^0(\mathcal U_x,\mathcal L_x^k)$ of the fiber of $f$ over $x\in \mathcal H$ is isomorphic to the fiber over the same point of the direct image $f_*(\mathcal L^k)$, at least for $k$ sufficiently large. Indeed, with this assumption, by flatness hypothesis the direct image $f_*(\mathcal L^k)$ is a locally free sheaf on $\mathcal H$ \cite[Proposition 7.9.13]{EGAIII2}, and by relative ampleness we have $H^q(\mathcal U_x,\mathcal L_x^k)=0$ for all $q>0$ so that the fiber $f_*(\mathcal L^k)_x$ is naturally isomorphic to $H^0(\mathcal U_x,\mathcal L_x^k)$ by \cite[Theorem 12.11]{H77} as claimed. Thus we conclude that 
\begin{eqnarray}\label{eq::rank=P} \rk f_*(\mathcal L^k) &=& \dim H^0(\mathcal U_x, \mathcal L_x^k), \\ \label{eq::det=det} \det f_*(\mathcal L^k)|_x &=& \det H^0(\mathcal U_x, \mathcal L_x^k) , \end{eqnarray} 
for all $x\in \mathcal H$ and $k \gg 0$. Since $\dim H^0(\mathcal U_x,\mathcal L_x^k)$ coincides with the Hilbert polynomial $\chi(k)$, we have an expansion 
\begin{equation}\label{eq::rk_exp} \rk f_*(\mathcal L^k) = a_0 k^n + a_1 k^{n-1} + \dots + a_n \quad \mbox{ for } k\gg 0.\end{equation} 
Now consider the determinant of the locally free sheaf $f_*(\mathcal L^k)$ for $k$ big enough. As an easy corollary of a result due to Knudsen and Mumford \cite[Proposition 4]{KnuMum76}, we have
\begin{equation}\label{eq::KM_exp} \det f_*(\mathcal L^k) = \mu_0^{k^{n+1}} \otimes \mu_1^{k^n} \otimes \dots \otimes \mu_{n+1}, \end{equation}
where $\mu_0, \dots, \mu_{n+1}$ are $\mathbb Q$-line bundles on $\mathcal H$. 

\begin{defn}
The Chow-line on the Hilbert scheme $\mathcal H$ is defined by $$\lambda_{Chow}(\mathcal H,\mathcal L) = \det f_*(\mathcal L) ^\frac{1}{\rk f_*(\mathcal L)} \otimes \mu_0^{-\frac{1}{a_0}}.$$
\end{defn}

Combining \eqref{eq::KM_exp} and \eqref{eq::rk_exp}, always for $k\gg 0 $, we get the asymptotic expansion 
\begin{eqnarray*}
\lambda_{Chow} (\mathcal H, \mathcal L^k) &=& \left( \det f_*(L^k) ^\frac{1}{a_0} \otimes \mu_0^{-\frac{k\chi(k)}{a_0^2}} \right)^\frac{a_0}{\chi(k)} \\
&=& \mu_{n+1}^\frac{1}{\chi(k)} \otimes \bigotimes_{\ell =1}^n \left( \mu_\ell ^\frac{1}{a_0} \otimes \mu_0^{-\frac{a_\ell}{a_0^2}} \right)^\frac{a_0 k^{n+1-\ell}}{\chi(k)}.
\end{eqnarray*}

\begin{defn}
The $\ell$-th CM-line on the Hilbert scheme $\mathcal H$ is defined by $$ \lambda_{\CM,\ell}(\mathcal H,\mathcal L) = \mu_\ell ^\frac{1}{a_0} \otimes \mu_0^{-\frac{a_\ell}{a_0^2}}.$$
\end{defn}

For $\ell=1$ we recover (a rational multiple of) the CM-line introduced by Paul and Tian \cite{PauTia06}.

Since $SL(N+1,\mathbb C)$ acts naturally on $\mathcal H$ and $\mathcal U$ making equivariant all maps introduced above, $\mathcal L$ comes equipped with a natural $SL(N+1,\mathbb C)$-linearization and expansion \eqref{eq::KM_exp} holds in the sense of linerized bundles. Hence the $\lambda_{\CM,\ell}$ are linearized bundles. Fix a one parameter subgroup of $SL(N+1,\mathbb C)$, and let $x\in \mathcal H$ be a fixed point, moreover set $V=\mathcal U_x$ and $A=\mathcal L_x$. Thus we are in the situation of definition \ref{defn::CM-weight} and are defined $b_\ell (V,A)$ as in \eqref{eq::asym_w(V,A^k)}. On the other hand by \eqref{eq::det=det} the Mumford weight of $\det f_*(\mathcal L^k)|_x$ is nothing but $w(V,A^k)$, thus by \eqref{eq::KM_exp} the Mumford weight of $\mu_\ell$ over $x$ must be $b_\ell$. This proves the following

\begin{prop}
The invariants $\Chow$ and $F_\ell$ are the Mumford weights of the lines $\lambda_{Chow}$ and $\lambda_{\CM,\ell}$ respectively.
\end{prop}


\section[Projective bundles over curves]{Projective bundles over curves}\label{sec::prob}

Let $\Sigma$ be a genus $g \geq 2$ smooth curve and let $E$ be a rank $n\geq 2$ vector bundle on $\Sigma$. Let $M=\mathbb P(E)$ be the projective bundle associated to $E$ and denote by $\pi: M \to \Sigma$ the projection. Each line bundle on $\mathbb P(E)$ has the form $L=\mathcal O_{\mathbb P(E)}(r) \otimes \pi^*B$, where $\mathcal O_{\mathbb P(E)}(1)$ the fiberwise hyperplane bundle, $B$ is a line bundles on $\Sigma$, and $r \in \mathbb Z$. Assume that $r$ and $B$ are such that $L$ is ample. The pair $(M,L)$ turns to be an $n$-dimensional polarized manifold, and if $E = E_1 \oplus \dots \oplus E_s$ is the (unique up to the order) decomposition in indecomposable vector bundles, then every $\mathbb C^*$-action on $\mathbb P(E)$ is induced by a fiberwise action on $E$ of the form $$ t \cdot (e_1, \dots, e_s) = (t^{\lambda_1}e_1, \dots, t^{\lambda_s}e_s),$$ with $\lambda_j \in \mathbb Z$. Finally let $\lambda_0$ be the weight of a fiberwise $\mathbb C^*$-action on $B$.

\begin{prop}\label{prop::Chow_PE}
In the situation above we have $$\Chow(M,L^k) = \frac{\binom{n-1+kr}{n}}{n+1}\frac{\chi\left(\Sigma,\det \left(E\otimes B^{-\frac{1}{r}}\right)\right)}{\mu\left(E\otimes B^{-\frac{1}{r}}\right) \chi\left(\Sigma,S^{kr}\left(E^*\otimes B^\frac{1}{r}\right)\right)} \sum_{j=1}^s \lambda_j \rank (E_j) \left( \mu(E_j)-\mu(E) \right),$$ where $\mu(F) = \deg(F) / \rank(F) $ is the slope of the bundle $F$, and.
\end{prop}

\begin{rem}
Although $E\otimes B^{-\frac{1}{r}}$ might not be a genuine vector bundle on $\Sigma$, its Chern classes are well defined as well as the formula above.
\end{rem}

\begin{proof}
By Definition \ref{defn::CM-weight} of Chow weight we need to calculate $\chi(M,L^k)$ and $w(M,L^k)$. By projection formula we have $\pi_* (L^k) \simeq S^{kr}E^* \otimes B^k$ equivariantly, thus $\chi(M,L^k) = \chi(\Sigma, S^{kr} E^* \otimes B^k)$ and $w(M,L^k) = w(\Sigma, S^{kr} E^* \otimes B^k)$.

Thus, by Hirzebruch-Riemann-Roch theorem and equalities $$\rank(S^{kr} E^*) = \binom{n-1+kr}{kr} \qquad \mbox{ and } \qquad c_1(S^{kr} E^*) = -\binom{n-1+kr}{kr-1} c_1(E)$$  we have
\begin{eqnarray*}
\chi(M, L^k) &=& \int_\Sigma \ch(S^{kr}E^*) \ch(B^k) \td(\Sigma) \\
&=& \int_\Sigma c_1(S^{kr}E^*) + \rk(S^{kr}E^*) \int_\Sigma c_1(B^k) + \rk(S^{kr}E^*) \int_\Sigma \frac{c_1(\Sigma)}{2}\\
\\
&=& -\binom{n-1+kr}{kr-1} \int_\Sigma c_1(E) + \binom{n-1+kr}{kr} \left( k \int_\Sigma c_1(B)+ \int_\Sigma \frac{c_1(\Sigma)}{2} \right) \\
&=& \binom{n-1+kr}{kr} \left( 1-g-kr \mu(E) + k \deg(B) \right).
\end{eqnarray*}

To calculate $w(M,L^k)$, which is equal to $w(\Sigma, S^{kr} E^* \otimes B^k)$ as noticed above, we use the Equivariant Hirzebruch-Riemann-Roch theorem. We have to integrate the degree $(2,2)$ component of the equivariant class $\ch(S^{kr} E^*) \ch(B^k) \td(\Sigma)$ (here we are using the same notations as non-equivariant case since we deal only with equivariant classes until the end of the proof). Since the given fiberwise actions on $E$ and $B$ cover the trivial action on $\Sigma$, the equivariant class $\td(\Sigma) = 1 + \frac{c_1(\Sigma)}{2}$ has no degree $(2,2)$ part. On the other hand by tedious but elementary calculation we get $$ \ch_2(S^{kr} E^*) = \binom{n+kr}{kr-1} \frac{c_1(E)^2 - 2c_2(E)}{2} + \binom{n-1+kr}{kr-2} \frac{c_1(E)^2}{2},$$
hence
\begin{eqnarray*}
w(M,L^k) &=& \int_\Sigma \ch_2(S^{kr} E^*) + \int_\Sigma c_1(S^{kr} E^*) \frac{c_1(\Sigma)}{2} + \int_\Sigma c_1(S^{kr} E^*) c_1(B^k) \\
&& + \rk(S^{kr}E^*) \int_\Sigma c_1(B^k)\frac{c_1(\Sigma)}{2} + \rk(S^{kr}E^*) \int_\Sigma \frac{c_1(B^k)^2}{2}\\ 
&=& \binom{n+kr}{kr-1} \int_\Sigma \frac{c_1(E)^2 - 2c_2(E)}{2} + \binom{n-1+kr}{kr-2} \int_\Sigma \frac{c_1(E)^2}{2} \\
&& - \binom{n-1+kr}{kr-1}\int_\Sigma \frac{c_1(E) c_1(\Sigma)}{2} - \binom{n-1+kr}{kr-1} k \int_\Sigma c_1(E) c_1(B) \\
&& + \binom{n-1+kr}{kr} k \int_\Sigma \frac{c_1(B)c_1(\Sigma)}{2} + \binom{n-1+kr}{kr} k^2 \int_\Sigma \frac{c_1(B)^2}{2}.
\end{eqnarray*} 

We will use the Cartan model of equivariant cohomology to calculate the equivariant characteristic numbers above. To this end, fix an $S^1$-invariant Hermitian metric on each $E_j$ of the decomposition $E=E_1\oplus\dots\oplus E_s$. The sum of such metrics gives an Hermitian metric on $E$, whose equivariant curvature of the compatible connection has the form $\Theta + \Lambda$, where $\Theta=\Theta_1\oplus\dots\oplus\Theta_s$ is the direct sum of non-equivariant curvatures of the fixed metrics on the $E_j$'s, and $\Lambda$ is the endomorphism of $E$ represented by $\lambda_1 I_1\oplus\dots\oplus\lambda_s I_s$, being $I_j$ the identity on $E_j$. By equivariant Chern-Weil theory we have $c_1(E) = [\tr (\Theta) + \tr (\Lambda)]$ and $c_1(E)^2 - 2 c_2(E) = [2\tr (\Theta\Lambda) + \tr(\Lambda^2)]$, and by construction $\tr(\Lambda) = \sum_{j=1}^s \lambda_j \rank(E_j)$ and $\tr(\Theta\Lambda) = \sum_{j=1}^s \lambda_j \tr(\Theta_j)$. In the same way, if we fix an Hermitian metric on $B$ with curvature $\omega$, we have $c_1(B)=[\omega+\lambda_0]$. We calculate

\begin{eqnarray*}
w(M,L^k) &=& \binom{n+kr}{kr-1} \int_\Sigma \tr(\Theta\Lambda) + \binom{n-1+kr}{kr-2} \tr(\Lambda) \int_\Sigma \tr(\Theta) \\
&& - \binom{n-1+kr}{kr-1}\tr(\Lambda) \int_\Sigma \frac{c_1(\Sigma)}{2} - \binom{n-1+kr}{kr-1} k \left( \lambda_0 \int_\Sigma \tr (\Theta) + \tr (\Lambda) \int_\Sigma \omega \right) \\
&& + \binom{n-1+kr}{kr} k \lambda_0 \int_\Sigma \frac{c_1(\Sigma)}{2} + \binom{n-1+kr}{kr} k^2 \lambda_0 \int_\Sigma \omega. \\
&=& \binom{n-1+kr}{kr}\left[ \frac{kr(n+kr)}{n(n+1)} \int_\Sigma \tr(\Theta\Lambda) + \frac{kr(kr-1)}{n+1} \tr(\Lambda) \mu(E) - \frac{kr}{n}\tr(\Lambda) (1-g) \right. \\
&& \left. - \frac{k^2r}{n} \tr(\Lambda) \int_\Sigma \omega - \frac{k^2r}{n} \lambda_0 \int_\Sigma \tr (\Theta) +k \lambda_0 (1-g) + k^2 \lambda_0 \int_\Sigma \omega. \right] \\
&=& \binom{n-1+kr}{kr}\left[ \frac{kr(n+kr)}{n(n+1)} \sum_{j=1}^s \lambda_j \rank (E_j) \left(\mu(E_j) - \mu(E)\right) \right. \\
&& \left. + k \left(\lambda_0 - \frac{r}{n}\tr(\Lambda) \right) (1-g-kr\mu(E) + k \deg(B))\right].
\end{eqnarray*}

By Definition \ref{defn::CM-weight}, subtracting from the ratio $w(M,L^k)/\chi(M,L^k)$ the linear term in $k$, we get
\begin{multline*} 
\Chow(M,L^k) = \\ \frac{kr}{n(n+1)} \frac{1-g+\deg(E)-\frac{n}{r}\deg(B)}{(\mu(E) - \frac{1}{r} \deg(B)) \left( 1-g-kr\mu(E)+k \deg(B) \right)} \sum_{j=1}^s \lambda_j \rank (E_j) \left(\mu(E_j) - \mu(E)\right),
\end{multline*} and the thesis follows.
\end{proof}

At this point we recall the following

\begin{defn}
A vector bundle $E$ over a smooth curve $\Sigma$ is called
\begin{itemize}
\item \emph{slope semistable} if $\mu(F) \leq \mu(E)$ for every subbundle $F \subset E$,
\item \emph{slope stable} if $\mu(F) < \mu(E)$ for every subbundle $F \subset E$,
\item \emph{slope polystable} if $E$ decomposes as $E=E_1 \oplus \dots \oplus E_s$ where the $E_j$'s are stable vector bundles satisfying $\mu(E_j)=\mu(E)$.
\end{itemize}
\end{defn}

We also recall that for every line bundle $B$ on $\Sigma$ we have $\mu(E\otimes B) = \mu(E) + \deg(B)$, thus $E$ is slope (semi/poly)stable if and only if $E \otimes B$ is.
 
Thanks to Proposition \ref{prop::Chow_PE} the invariants $F_\ell$ defined in the statement of Proposition \ref{prop::higher_Futaki} can be easily calculated in the considered case.

\begin{cor}\label{cor::Futaki=Chow} 
In the situation above we have $\Chow(M,L^k)=0$ if and only if $F_\ell(M,L)=0$ for some (and hence any) $\ell=1,\dots,n$. More precisely there are rational constants $C_\ell>0$, depending only on  $n$, such that
\begin{equation}\label{eq::F_ell} F_\ell (M,L) = - C_\ell \frac{\chi(\Sigma,\det (E \otimes B^{-\frac{1}{r}}))}{\mu(E \otimes B^{-\frac{1}{r}})^2} \sum_{j=1}^s \lambda_j \rank (E_j) \left( \mu(E_j)-\mu(E) \right). \end{equation}
In particular $F_\ell(M,L)=0$ for all $\mathbb C^*$-actions if and only if $ \mu(E_j) = \mu(E)$ for all $j = 1, \dots, s.$
\end{cor}

\begin{proof}
By definition, $F_\ell(M,L)$ is the coefficient of $k^{n+1-\ell}$ in $\frac{\chi(M,L^k)}{a_0(M,L)}\Chow(M,L^k)$. Thanks to Proposition \ref{prop::Chow_PE} (and its proof) we have $a_0(M,L) = -\frac{\mu(E \otimes B^{-\frac{1}{r}})}{(n-1)!}$ and
$$ \frac{\chi(M,L^k)}{a_0(M,L)}\Chow(M,L^k) = -\frac{(n-1)!}{n+1}\binom{n-1+k}{n}\frac{\chi(\Sigma,\det E \otimes B^{-\frac{1}{r}})}{\mu(E \otimes B^{-\frac{1}{r}})^2} \sum_{j=1}^s \lambda_j \rank (E_j) \left( \mu(E_j)-\mu(E) \right). $$
Thus \eqref{eq::F_ell} is proved with constants $C_\ell$ generated by $$ \sum_{\ell=1}^n C_\ell \, k^{n+1-\ell} = \frac{1}{n(n+1)} \prod_{i=0}^{n-1} (k+i).$$

To prove the last assertion it remains to show that $\chi(\Sigma,\det (E\otimes B^{-\frac{1}{r}}))$ is non-zero. Ampleness of $L$ implies $\chi(M,L^k)>0$ for $k\gg 0$, but on the other hand $\chi(M,L^k) = -\frac{\deg(E \otimes B^{-\frac{1}{r}})}{n!}k^n + O(k^{n-1})$, thus we get $\deg(E \otimes B^{-\frac{1}{r}})<0$. By that and $g\geq 2$ we obtain $$\chi(\Sigma,\det(E\otimes B^{-\frac{1}{r}})) = \deg(E\otimes B^{-\frac{1}{r}}) +1-g <0,$$ and we are done.\end{proof}

The relation between slope stability of $E$ and Chow (and Hilbert) stability of $\mathbb P(E)$ has been initially investigated by Morrison \cite{Mor80} when $\rank(E)=2$. He conjectured that the Chow and Hilbert stability of $\mathbb P(E)$ are equivalent to slope stability of $E$ \cite[Conjecture 5.10]{Mor80}. Combining theorems above and results of Mabuchi \cite{Mab05}, we get the following result which proves the Morrison's conjecture for $E$ of every rank, with the extra hypothesis that the chosen polarization on $\mathbb P(E)$ is sufficiently divisible.

\begin{thm}\label{thm::Morrison_conj}
A projective bundle $\mathbb P(E)$ over a curve of genus $g \geq 2$ is asymptotically Chow (poly)stable with respect to some (and hence any) polarization if and only if $E$ is slope (poly)stable.  
\end{thm}

\begin{proof}
Assume that $E$ is slope polystable. In this case from a well-known result of Narasimhan and Seshadri \cite{NarSes65} it follows that $\mathbb P(E)$ admits a cscK metric in each K\"ahler class. Moreover, by Corollary \ref{cor::Futaki=Chow} all the invariants $F_\ell$ are zero. Thus by Mabuchi's Theorem \ref{thm::Mabuchi}, $\mathbb P(E)$ is asymptotically Chow polystable. In particular, if $E$ is stable then it is indecomposable and $\mathbb P(E)$ admits no nontrivial $\mathbb C^*$-actions and it is stable.

Now assume that $E$ is not slope polystable. This means that we have a subbundle $F \subset E$ such that either $\mu(F) > \mu(E)$ and $E=F \oplus E/F$, or $\mu(F) \geq \mu(E)$ and $F$ is not a direct summand. In the first case, $\mathbb P(E)$ is not asymptotically Chow polystable since, taking the $\mathbb C^*$-action defined by $\diag(t,1)$, the (asymptotic) Chow weight is non-zero thanks to Proposition \ref{prop::Chow_PE}. In the second case, since $F$ is not a direct summand of $E$, the exact sequence 
$$ 0 \to F \to E \to E/F \to 0$$ 
gives a non-zero element $e \in \Ext^1(E/F,F)$ by which one can construct a non product test configuration for $\mathbb P(E)$ with central fiber isomorphic to $\mathbb P(F \oplus E/F)$ and non-negative Chow weight (compare with Ross and Thomas \cite[Remark 5.14]{RosTho06}). We need to make explicit this construction to find the induced action on the central fiber and then compute the Chow weight. To this end, let $F \to (I^\bullet, d)$ be an injective resolution of $F$, so that $e$ is represented by a map $\epsilon: E/F \to I^1$ such that $d \circ \epsilon = 0$, and for all $z \in \mathbb C$ the element $ze$ is represented by the map $z\epsilon$. Thanks to the assumptions the sequence
$$ 0 \to F \to I^0 \oplus E/F \to E/F \to 0,$$ 
where the second arrow is the inclusion in $I^0$ and the third is the projection on $E/F$, is exact. Now consider the map $\psi: I^0 \oplus E/F \oplus \mathbb C \to I^1$ defined by $\psi(a,u,z) = d(a)+z\epsilon(u)$ and let $\mathcal E$ be its kernel. Clearly $\mathcal E$ projects over $E/F \oplus \mathbb C$, and for all $(f,z) \in F\oplus \mathbb C$ we have $\psi(f,0,z) =0$, thus we have the sequence 
$$ 0 \to F \oplus \mathbb C \to \mathcal E \to E/F \oplus \mathbb C \to 0, $$
which restrict for every $z \in \mathbb C$ to the extension of $E/F$ by $F$ represented by $ze \in \Ext^1(E/F,F)$. Since $\mathcal E$ is invariant for the $\mathbb C^*$-action on $I^0 \oplus E/F \oplus \mathbb C$ defined by $t \cdot (a,u,z) = (ta,u,tz)$, the projectivization $\mathbb P(\mathcal E) \to \mathbb C$ gives a non-product test configuration for $\mathbb P(E)$ with central fiber $\mathbb P(F \oplus E/F)$ acted on by $\mathbb C^*$ via $\diag(t,1)$. Thanks to Proposition \ref{prop::Chow_PE} the Chow weight of that configuration is non-negative, thus $\mathbb P(E)$ is not polystable.
\end{proof}

The slope stability of $E$ is related to the existence of cscK metrics on $\mathbb P(E)$ as stated in the following  

\begin{thm}[Apostolov, Calderbank, Gauduchon and T{\o}nnesen-Friedman \cite{ApoCalGauTon09}]\label{thm::Gaud_cscK} 
A projective bundle $\mathbb P(E)$ over a curve of genus $g\geq 2$ admits a cscK metric in some (and hence any) K\"ahler class if and only if $E$ is polystable. 
\end{thm}

Thus by theorems \ref{thm::Morrison_conj} and \ref{thm::Gaud_cscK} we get the following

\begin{thm}\label{thm::Chow_cscK}
A projective bundle $\mathbb P(E)$ over a curve of genus $g \geq 2$ is cscK if and only if it is asymptotically Chow polystable w.r.t. some (and hence any) polarization.
\end{thm}

Theorem above gives a new proof of the Yau-Tian-Donaldson conjecture for projective bundles once one shows that $K$-(poly)stability is equivalent to asymptotic Chow (poly)stability in this case. This is a subtle purely algebraic problem. On the other hand we notice that in this case the YTD conjecture follows by known results. Indeed by Ross and Thomas \cite[Theorem 5.13]{RosTho06} and Theorem \ref{thm::Gaud_cscK} the K-polystability of $\mathbb P(E)$ implies the existence of a cscK metric, and by \cite{Mab08, Mab09} if $\mathbb P(E)$ admits a cscK metric, then it is $K$-polystable. 


\section[Blow-up of points]{Blow-up of points}\label{sec::blup}

Let $(M,L)$ be a polarized manifold of dimension $n \geq 2$. Let $S=\{p_1,\dots,p_s\} \subset M$ be a set of points and consider the blow-up $$\beta : \tilde M \to M$$ of $M$ at $S$ with exceptional divisors $E_j = \beta^{-1}(p_j)$. Let $m, \alpha_1,\dots,\alpha_s$ be positive integers such that the line bundle $\tilde L = \beta^*L^m \otimes \mathcal O(-\sum_{j=1}^s \alpha_jE_j)$ is ample on $\tilde M$. Moreover suppose given a $\mathbb C^*$-action on $M$ that fixes every $p_j$. This induces an action on $\tilde M$ that makes $\beta$ an equivariant map, and it is known that any $\mathbb C^*$-action on $\tilde M$ has this form. Since each exceptional divisor $E_j$ is invariant, the associated line bundle $\mathcal O(-E_j)$ has a natural linearization, so that $\tilde L$ is a linearized line bundle. Finally let $\omega$ be a $S^1$-invariant K\"ahler form representing $c_1(L)$, and let $\phi$ be a moment map for $\omega$. The choice of $\phi$ is equivalent to fix a linearization of the given action on $L$. Finally let $a_\ell(M,L)$ and $b_\ell(M,L)$ be defined as in \eqref{eq::asym_chi(V,A^k)} and \eqref{eq::asym_w(V,A^k)}.

\begin{thm}\label{thm::Fl_blowup}
In the situation above, assume that $(M,L)$ is asymptotically Chow polystable. Let $\lambda(p_j)$ be the total weight of the isotropy action on $T_{p_j}M$. We have 
$$ \Chow(\tilde M, \tilde L^k) = \frac{a_0(M,L^m) - \sum_{j=1}^s \alpha_j^n/n! }{\chi(M,L^{mk})- \sum_{j=1}^s \binom{n+\alpha_jk-1}{\alpha_jk-1}} \sum_{\ell=1}^n F_\ell(\tilde M, \tilde L) k^{n+1-\ell},$$
where  
\begin{equation}\label{eq::Fell_blup} 
F_\ell(\tilde M, \tilde L) = \frac{1}{D^2m^{\ell-1}} \sum_{j=1}^s \left( f_\ell \left( \frac{\alpha_j}{m} \right) \phi(p_j) - g_\ell \left( \frac{\alpha_j}{m} \right) \lambda(p_j) \right),
\end{equation}
$ D = \deg(M,L) - \sum_{j=1}^s \left(\frac{\alpha_j}{m}\right)^n > 0$, and $f_\ell, g_\ell$ are polynomial functions defined by
$$
f_\ell \left( x \right) = D s_{n-\ell} x^{n-\ell} - \left( n!\, a_\ell(M,L) - s_{n-\ell} \sum_{i=1}^s \left(\frac{\alpha_i}{m}\right)^{n-\ell} \right) x^n,
$$
$$
g_\ell(x) = \frac{D s_{n+1-\ell} x^{n+1-\ell} - x f_\ell(x)}{n+1},
$$
and $s_h$ are non-negative rational numbers, depending on $n$, generated by $\binom{n+x-1}{x-1} = \sum_{h \geq 0} \frac{s_h x^h}{n!}$.
\end{thm}

\begin{proof}
We start fixing some notations. Let $\mathcal I_j \subset \mathcal O_M$ be the ideal sheaf of the point $p_j$, and let $Z\subset M$ be the subscheme cut out by $\mathcal I_Z = \mathcal I_1^{\alpha_1}\cdots \mathcal I_s^{\alpha_s}$.  For $k$ sufficiently large we have $\beta_*\mathcal O(-k\sum_{j=1}^s \alpha_j E_j) = \mathcal I_Z^k$ and $R^i \beta_*\mathcal O(-k\sum_{j=1}^s \alpha_j E_j)=0$ for all $i>0$. Thus, by projection formula we get $H^0(\tilde M, \tilde L^k) \simeq  H^0(M,\mathcal I_Z^k \otimes L^{mk})$ equivariantly and $H^i(M,\mathcal I_Z^k \otimes L^{mk})=0$ for all $i>0$, hence we have to calculate the dimension and the weight of $H^0(M,\mathcal I_Z^k \otimes L^{mk})$. To this end we consider the equivariant exact sequence 
$$ 0 \to \mathcal I_Z^k \otimes L^{mk} \to L^{mk} \to L^{mk} / \mathcal I_Z^k \otimes L^{mk} \to 0,$$ 
and the induced sequence \begin{equation}\label{eq::exev_seq} 0 \to H^0(M,\mathcal I_Z^k \otimes L^{mk}) \to H^0(M,L^{mk}) \to H^0(M, L^{mk} / \mathcal I_Z^k \otimes L^{mk}) \to 0,\end{equation} which is equivariant and exact by considerations above. The sheaf  $L^{mk} / \mathcal I_Z^k \otimes L^{mk}$ is supported on $S$, and we have the equivariant isomorphism
$$ H^0(M,L^{mk} / \mathcal I_Z^k \otimes L^{mk}) \simeq \bigoplus_{j=1}^s L_{p_j}^{mk} \otimes \left( \bigoplus_{d=0}^{\alpha_j k - 1} \sym^d T^*_{p_j}M\right), $$
 where $\mathbb C^*$ acts on $L_{p_j}$ with weight $-\phi(p_j)$ and on $T^*_{p_j}M$ with total weight $-\lambda(p_j)$. Thus, by easy calculations we get 
$$ \chi(M,L^{mk} / \mathcal I_Z^k \otimes L^{mk}) = \sum_{j=1}^s \binom{n+\alpha_j k-1}{\alpha_j k-1}, $$ 
and 
\begin{equation}\label{eq::wei_quot}
w(M,L^{mk} / \mathcal I_Z^k \otimes L^{mk}) = -\sum_{j=1}^s \left[\binom{n+\alpha_j k-1}{\alpha_j k-1}  mk\,\phi(p_j) + \binom{n+\alpha_j k -1}{\alpha_j k -2} \lambda(p_j) \right].
\end{equation}

Moreover, by \eqref{eq::exev_seq} and the ampleness of $L$ and $\tilde L$ we obtain 
\begin{eqnarray}
\label{eq::chi_blup1} \chi(\tilde M, \tilde L^k) &=& \chi(M,L^{mk}) - \sum_{j=1}^s \binom{n+\alpha_j k-1}{\alpha_j k-1} \\
\nonumber &=& \sum_{\ell=0}^n a_\ell(M,L) (mk)^{n-\ell} - \sum_{j=1}^s \sum_{h=0}^n  \frac{s_h}{n!} (\alpha_j k)^h \\
\label{eq::chi_blup2} &=& \sum_{\ell=0}^n \left( a_\ell(M,L)m^{n-\ell} - \frac{s_{n-\ell}}{n!}\sum_{j=1}^s \alpha_j^{n-\ell} \right) k^{n-\ell}.
\end{eqnarray}

Since $(M,L)$ is asymptotically Chow polystable, we can fix a linearization on $L$ such that $b_\ell(M,L)=0$ for all $\ell=0,\dots,n$. Indeed, we can choose the moment map $\phi$ such that $b_0(M,L) = \int_M \phi \, \omega^{n}/n!=0$, and in this case the asymptotic Chow polystability implies $w(M,L^k)=0$ for all $k \gg 0$. Thus by \eqref{eq::exev_seq} and \eqref{eq::wei_quot} we have
\begin{eqnarray}
\nonumber w(\tilde M, \tilde L^k) &=& \sum_{j=1}^s \left[\binom{n+\alpha_j k-1}{\alpha_j k-1}  mk\,\phi(p_j) + \binom{n+\alpha_j k -1}{\alpha_j k -2} \lambda(p_j) \right]\\
\label{eq::w_blup} &=& \sum_{\ell=0}^n \left( \frac{s_{n-\ell}}{n!} m \sum_{j=1}^s \alpha_j^{n-\ell} \phi(p_j) + \frac{s_{n-\ell}-s_{n+1-\ell}}{(n+1)!} \sum_{j=1}^s \alpha_j^{n+1-\ell} \lambda(p_j) \right) k^{n+1-\ell}.
\end{eqnarray}

Thanks to smoothness of $\tilde M$, the first formula of the statement follows easily from definition of $\Chow(\tilde M, \tilde L^k)$ and $F_\ell(\tilde M, \tilde L^k)$, and the identity
$$ 
\frac{a_0(\tilde M, \tilde L)}{\chi(\tilde M, \tilde L^k)} = \frac{a_0(M,L^m) - \sum_{j=1}^s \alpha_j^n / n!}{\chi(M,L^{mk})- \sum_{j=1}^s \binom{n+\alpha_jk-1}{\alpha_jk-1}},
$$
which in turn is a straightforward consequence of equations \eqref{eq::chi_blup1}, \eqref{eq::chi_blup2} and $s_n=1$.

Moreover, in order to get formula \eqref{eq::Fell_blup}, we calculate 
\begin{eqnarray*}
F_\ell(\tilde M, \tilde L^k) &=& \frac{a_0(\tilde M, \tilde L) b_\ell(\tilde M, \tilde L) - a_\ell(\tilde M, \tilde L) b_0(\tilde M, \tilde L)}{a_0(\tilde M, \tilde L)^2} \\
&=& \left( \frac{1}{\deg(M,L) -\sum_{j=1}^s \left( \frac{\alpha_j}{m} \right)^n } \right)^2 \frac{a_0(\tilde M, \tilde L) b_\ell(\tilde M, \tilde L) - a_\ell(\tilde M, \tilde L) b_0(\tilde M, \tilde  L)}{(m^n/n!)^2}
\end{eqnarray*}
\begin{eqnarray*}
&=& D^{-2} \left[ \left( \deg(M,L) - \sum_{i=1}^s \left(\frac{\alpha_i}{m}\right)^n \right) \left( s_{n-\ell} \sum_{j=1}^s \left(\frac{\alpha_j}{m}\right)^{n-\ell} \phi(p_j) + \right. \right. \\
&&+\left. \frac{s_{n-\ell}-s_{n+1-\ell}}{n+1} \sum_{j=1}^s \left(\frac{\alpha_j}{m}\right)^{n+1-\ell} \lambda(p_j) \right) - \left( n!\,a_\ell(M,L) - s_{n-\ell} \sum_{i=1}^s \left(\frac{\alpha_i}{m}\right)^{n-\ell} \right) \times \\
&& \times \left. \left( \sum_{j=1}^s \left(\frac{\alpha_j}{m}\right)^n \phi(p_j) + \frac{1}{n+1} \sum_{j=1}^s \left(\frac{\alpha_j}{m}\right)^{n+1} \lambda(p_j) \right) \right] m^{1-\ell} \\
&=& D^{-2} \left( \sum_{j=1}^s f_\ell \left(\frac{\alpha_j}{m}\right)\phi(p_j) - \sum_{j=1}^s g_\ell \left(\frac{\alpha_j}{m}\right)\lambda(p_j) \right)m^{1-\ell},
\end{eqnarray*}
whence \eqref{eq::Fell_blup} follows easily.
\end{proof}

\begin{rem}
From formula \eqref{eq::Fell_blup} follows easily the asymptotic of the invariants $F_\ell$ in the so-called ``adiabatic limit'', $m \to \infty$. In particular, for $\ell =1$ we have 
$$
F_1(\tilde M, \tilde L) = \frac{n(n-1)}{2 \deg(M,L)} \sum_{j=1}^s \left(\frac{\alpha_ j}{m}\right)^{n-1} \phi(p_j) + O \left(\frac{1}{m^n}\right).
$$

For each invariant point $p \in S$, the Chow weight of the zero dimensional polarized scheme $(p,L_p)$ is clearly zero by definition. On the other hand, we define 
$$
w_\CW (p,L) = \phi(p) - \frac{b_0(M,L)}{a_0(M,L)},
$$
and we extend by linearity to each zero-cycle $\sum_{j=1}^s c_j p_j$, with $c_j \in \mathbb Q$. Thus, recalling that $(M,L)$ is assumed to be asymptotically Chow polystable, we get
$$
F_1(\tilde M, \tilde L) = \frac{n(n-1)}{2 \deg(M,L)} w_\CW\left(  \sum_{j=1}^s \alpha_ j^{n-1} p_j,L\right) \frac{1}{m^{n-1}} + O \left(\frac{1}{m^n}\right).
$$
The same formula has been proved by Stoppa for a wider class of test configurations of blown-up manifolds \cite{Sto10}. 
\end{rem}

Theorem \ref{thm::Fl_blowup} suggests that asymptotically Chow polystability is in general lost after blowing up if some non-trivial $\mathbb C^*$-action still remains. Indeed it is easy to find $\alpha_1,\dots,\alpha_s$ and $m$ such that $F_\ell(\tilde M, \tilde L)\neq 0$, at least for some $\ell$. Refining this simple remark we can show by an example that blowing up a cscK asymptotically Chow polystable manifold can produce integral cscK classes which are asymptotically Chow unstable.

The strategy is the following. Given a K\"ahler manifold $M$, we know that the extremal cone is open in the K\"ahler cone by Lebrun and Simanca \cite{LebSim94}, and every extremal class with zero Futaki invariant is in fact cscK. Thus the set $\mathcal C$ of the cscK classes is open in the locus $\mathcal F$ where $F_1=0$. Moreover, when $M$ is a cscK manifold blown-up at some points, $\mathcal C$ is non-empty under mild conditions by Arezzo and Pacard \cite{ArePac09}. If we can prove that the locus $\mathcal Z$ where $F_1=\dots=F_n=0$ is a proper Zariski closed of $\mathcal F$, and $\mathcal C \setminus \mathcal Z$ has at least a rational point, we are done.
Clearly we look for a rational class in $\mathcal C \setminus \mathcal Z$ because it has to be (a rational multiple of) the first Chern class of a line bundle. In the example below we will show that rational points of $\mathcal F$ are in fact dense in the Euclidean topology.

\begin{ex}
Let $(M,L)$ be the complex projective plane $\mathbb P^2$ polarized with the hyperplane bundle, and let $S=\{p_1, p_2, p_3, p_4\}$ be a set of points of $M$, and suppose that the $p_j$'s are all but one aligned. Up to a projectivity we can suppose $p_1=(1,0,0)$, $p_2=(0,1,0)$, $p_3=(0,a,b)$, and $p_4=(0,a,-b)$ with $ab\neq0$.
Let $(\tilde M, \tilde L)$ be the blow-up of $M$ along $S$ endowed with the polarization defined as in the general case above $\tilde L=\beta^*L^m \otimes \mathcal O\left(-\sum_{j=1}^4 \alpha_j E_j\right)$. 

The unique (up to scaling) holomorphic vector field of $\tilde M$ generates the $\mathbb C^*$-action ${\rm diag}(t^2,t^{-1},t^{-1})$ on $\mathbb P^2$, thus the condition $F_\ell(\tilde M, \tilde L)=0$ in the K\"ahler cone is equivalent to a (homogeneous polynomial) equation in the variables $(m,\alpha)$.

By direct calculations from formula \eqref{eq::Fell_blup} we get $F_1(\tilde M, \tilde L)=\deg(\tilde M, \tilde L)^{-2} \psi_1(m,\alpha)$, where
\begin{multline}\label{eq::f=0}
\psi_1(m,\alpha)=(2\alpha_1-\alpha_2-\alpha_3-\alpha_4)(m^3-3\alpha_1^2m) -(2\alpha_1^2-\alpha_2^2-\alpha_3^2-\alpha_4^2) (3m^2-3\alpha_1m)+\\
+( 2\alpha_1^3-\alpha_2^3-\alpha_3^3-\alpha_4^3)(3m-\alpha_1-\alpha_2-\alpha_3-\alpha_4),
\end{multline}
and $F_2(\tilde M, \tilde L)=\deg(\tilde M, \tilde L)^{-2} \psi_2(m,\alpha)$, where
\begin{multline}
\psi_2(m,\alpha)=(2\alpha_1-\alpha_2-\alpha_3-\alpha_4)(m^2-\alpha_1^2-\alpha_2^2-\alpha_3^2-\alpha_4^2)+\\
-2(2\alpha_1^2-\alpha_2^2-\alpha_3^2-\alpha_4^2)m+2( 2\alpha_1^3-\alpha_2^3-\alpha_3^3-\alpha_4^3).
\end{multline}
Since $\psi_1$ and $\psi_2$ are irreducible and in particular $\psi_2$ does not divide $\psi_1$, the locus $\psi_2\neq 0$ is non empty and Zariski open in the locus $\psi_1=0$. On the other hand, is easy to see that $(1,1,0,0,0) \in \mathbb P^4$ is a triple point for the quartic threefold defined by $\psi_1=0$, thus the set of rational solutions of the equation $\psi_1(m,\alpha)=0$ is dense in the euclidean topology in the locus $\psi_1=0$. Since the latter is non empty by Arezzo and Pacard \cite[Example 7.3]{ArePac09}, we get that there are infinitely many integral cscK classes on $\tilde M$ which are asymptotically Chow unstable.
\end{ex}

\begin{ex}
The argument above shows that all rational cscK classes on the blow-up of $\mathbb P^2$ along a set $S=\{p_1, p_2, p_3\}$ of three non-aligned points are asymptotically Chow stable. Indeed in this case the null loci of $F_1$ and $F_2$ are the same.

More precisely, assume without loss of generality that $p_1=(1,0,0)$, $p_2=(0,1,0)$, $p_3=(0,0,1)$ and let $(\tilde M, \tilde L)$ be the blow-up of $\mathbb P^2$ along $S$ endowed with the polarization $\tilde L$ defined as above.
A set of generators for the space of holomorphic vector fields of $\tilde M$ is now given by the fields associated with the $\mathbb C^*$-actions ${\rm
diag}(t,t^{-1}, 1)$ and ${\rm diag}(1,t,t^{-1})$ on $\mathbb P^2$. After a straight calculation involving formula \eqref{eq::Fell_blup}, we see that $F_1(\tilde M, \tilde L)=0$ if and only if 
$$
\left\{
\begin{array}{lll}
(\alpha_1 - \alpha_2) (m-\alpha_1-\alpha_2-\alpha_3)(m^2 - 2 \alpha_1 m - 2 \alpha_2 m + \alpha_3 m + \alpha_1^2 + \alpha_1 \alpha_2 + \alpha_2^2) = 0 \\
(\alpha_2 - \alpha_3) (m-\alpha_1-\alpha_2-\alpha_3)(m^2 - 2 \alpha_3 m - 2 \alpha_2 m + \alpha_1 m + \alpha_3^2 + \alpha_3 \alpha_2 + \alpha_2^2) = 0
\end{array}
\right.
$$

while $F_2(\tilde M, \tilde L)=0$ is equivalent to
$$
\left\{
\begin{array}{lll}
(\alpha_1 - \alpha_2) (m-\alpha_1-\alpha_2-\alpha_3)(m - \alpha_1 - \alpha_2 + \alpha_3) = 0 \\
(\alpha_2 - \alpha_3) (m-\alpha_1-\alpha_2-\alpha_3)(m + \alpha_1 - \alpha_2 - \alpha_3) = 0
\end{array}
\right.
$$

In order for $L_m$ to be ample, by Nakai's criterion, it must be $m>0$, $\alpha_i >0$, and $m + \alpha_j>\alpha_1+\alpha_2+\alpha_3$ for all $j=1,2,3$, thus one immediately sees that $F_1(\tilde M, \tilde L)=0$ if and only if $F_2(\tilde M, \tilde L)=0$, and both are zero if and only if either $\alpha_1=\alpha_2=\alpha_3$ or $\alpha_1+\alpha_2+\alpha_3=m$ (compare with \cite[Example 3.2]{LebSim94}).
\end{ex}

\end{document}